\documentclass[11pt,twoside]{article}
\usepackage{latexsym}
\usepackage{amssymb,amsbsy,amsmath,amsfonts,amssymb,amscd,amsthm}
\usepackage{epsfig, graphicx}
\usepackage{color}
\usepackage{enumerate}
\usepackage{pdfsync}
\usepackage{url}
\usepackage{xypic}
\newcommand{\R}{\mathbb{R}}
\newcommand{\N}{\mathbb{N}}

\def\a{\alpha}
\def\b{\beta}
\def\d{\delta}
\def\D{\Delta}
\def\D{\Delta}
\def\g{\gamma}

\def\l{\lambda}

\def\e{\varepsilon}

\def\pd{\partial}
\def\half{\frac{1}{2}}

\newcommand{\cF}{{\cal F}}

\newcommand{\cM}{{\mathcal M}}

\newcommand{\cS}{{\cal S}}

\newcommand{\f}{\frac}
\newcommand{\Lip}{{\rm Lip}}

\newcommand{\AM}{{\rm AM}}
\newcommand{\AML}{{\rm AMLE}}
\newtheorem{theorem}{Theorem}[section]
\newtheorem{lemma}[theorem]{Lemma}
\newtheorem{definition}[theorem]{Definition}
\newtheorem{proposition}[theorem]{Proposition}
\newtheorem{corollary}[theorem]{Corollary}
\newtheorem{remark}[theorem]{Remark}
\numberwithin{equation}{section}
\title{\Large \bf Absolutely Minimizing Lipschitz Extensions
        and Infinity Harmonic Functions on the Sierpinski gasket}
\author{Fabio Camilli\footnotemark[1] \and Raffaela Capitanelli\footnotemark[1]
        \and Maria Agostina Vivaldi\footnotemark[1]}
\date{\today}
\begin{document}
\maketitle
\footnotetext[1]{Dip. di Scienze di Base e Applicate per l'Ingegneria,  ``Sapienza'' Universit{\`a}  di Roma, via Scarpa 16,
 00161 Roma, Italy, ({\tt e-mail:camilli,capitanelli,vivaldi@sbai.uniroma1.it})}
\pagestyle{plain}
\pagenumbering{arabic}
\begin{abstract}
Aim of this note is to study the infinity Laplace operator and the corresponding Absolutely Minimizing Lipschitz Extension problem on the Sierpinski gasket in the spirit of the classical construction of Kigami for the Laplacian. We introduce a notion of infinity harmonic functions on   pre-fractal sets
and we show that these functions solve a  Lipschitz extension problem in the discrete setting. Then we  prove that the limit   of the infinity harmonic functions on the pre-fractal sets solves the    Absolutely Minimizing Lipschitz Extension problem on the Sierpinski gasket.
\end{abstract}
 \begin{description}
\item [{\bf MSC2010}:]  31C20, 28A80.
\item [{\bf Keywords}:] Sierpinski gasket, McShane-Whitney extensions, Absolute Minimizing Lipschitz Extension,   infinity Laplacian.
\end{description}
\section{Introduction}
The theory of Absolutely Minimizing Lipschitz Extension (\cite{acj,c}) concerns the classical problem of extending a Lipschitz continuous function
 $f$ defined on the boundary of an open set $U\subset\R^d$ to the interior of $U$ without increasing the Lipschitz constant. In other words,  to find a Lipschitz continuous function $u:\overline U\to\R$ such that $u=f$ on $\pd U$ and  $\Lip(u, U)=\Lip(f,\pd U)$ ($\Lip$ denotes the  Lipschitz constant).\par
The previous problem has several solutions, all in between a maximal  and a minimal  one called  respectively the  McShane's extension and the Whitney's extension.  But, among all these possible solutions, the ``canonical" one   is the so called  Absolutely Minimizing Lipschitz Extension (AMLE in short). This function is characterized by satisfying  the extension problem  not only in  $U$, but  also  in any  open subset $V$ of $U$, that is  $\Lip(u, V)=\Lip(u,\pd V)$ for any  open   $V\subset U$. The relevance  of the notion of AMLE relies in the several additional properties that this function satisfies, for  example
it is  the unique viscosity solution of the Dirichlet problem for the infinity Laplace equation
\begin{equation}\label{intro1}
\Delta_\infty u(x)=0, \qquad x\in U.
\end{equation}
Here $\D_\infty u=\sum_{i,j=1}^d\pd_{x_i}u\pd_{x_j} u\pd^2_{x_ix_j}u$, the infinity Laplacian,  is a  nonlinear degenerate second order operator and
a function satisfying \eqref{intro1} is said infinity harmonic.\par
The theory of AMLE has been extended to lenght spaces   (see  \cite{acj,cdp,j,js, pssw}), hence it applies in particular
to the Sierpinski gasket $\cS$ endowed with its geodesic distance. Prescribed  a boundary data $f$ on the vertices of the initial simplex from which $\cS$ is obtained via iteration, there exists   a unique    AMLE of $f$ to $\cS$. Moreover this function can be characterized as in the Euclidean case by an intuitive geometric property, called  Comparison with Cones.  \par
After the seminal work of Kigami   \cite{k}, a standard way to define a harmonic function on
 the Sierpinski gasket, and more in general for the class of post-critically finite fractals,  is to consider      the uniform limit
of solutions of suitable scaled finite difference differences on the pre-fractals.

%
For the infinity Laplace operator this approach   has been  pursued in \cite{g}. In this thesis, a graph infinity Laplacian on pre-fractal sets is defined and an algorithm to compute explicitly
infinity harmonic functions   is studied. By means of a constructive approach based on the previous algorithm,  it is proved that   the sequence of the infinity harmonic functions on the pre-fractal sets converges to a function defined on the limit fractal set. It is worth noting  that the same graph  infinity Laplacian is used in the Euclidean case to approximate the viscosity solution of \eqref{intro1} (see \cite{lg,mos,o}).\par
Following an approach  similar to \cite{g}, we aim to define an  infinity harmonic  function on $\cS$ as the limit of solutions of  finite difference equations on  pre-fractals. We  study  the Lipschitz extension problem on  the  pre-fractal sets and we show that an appropriate   notion of AMLE  can be introduced in this framework. We also prove that an AMLE is a  solution of the graph infinity Laplacian and it    satisfies a Comparison with  Cone property
with respect to the path distance. The Comparison with  Cone   property on the  pre-fractals is crucial since it  allows us to show that the limit of the AMLEs on the  pre-fractals is an AMLE on the Sierpinski.   The convergence result allows us to define
an infinity harmonic function on the Sierpinski as the limit of infinity harmonic functions on the pre-fractals and to conclude the equivalence, as in
Euclidean case, between  AMLE and infinity harmonic functions. Hence, besides  giving a simpler proof of the convergence
result   in \cite{g}, we also obtain as in the Euclidean case   the equivalence among the various properties which characterize the AMLE theory.
We   remark that the previous construction   on  the Sierpinski gasket can be readily extended
to the class of  post-critically finite fractals since it is only based on the convergence of the path distance defined  on the pre-fractal sets
to the path distance (intrinsic length) on the limit fractal set (\cite{ccm}, \cite{hs}).\par
The paper is organized as follows. In Section \ref{sec2} we introduce   notations and definitions and we prove some preliminary properties
    for the graph infinity Laplacian. Section  \ref{sec3} is devoted to the AMLE problem on pre-fractals. In Section  \ref{sec4} we recall the
definition of AMLE in metric spaces and we prove the convergence result for the pre-fractals invading  the Sierpinski gasket. Finally, in Section  \ref{sec5}  we describe the algorithm in \cite{g} and we study the relation between infinity and   $p$-harmonic function in the pre-fractals.

\section{Notations and preliminary results}\label{sec2}
In this section we introduce   notations and definitions  and we collect some preliminary results
on infinity harmonic functions on graphs.\par
Consider an unitary equilateral triangle $V^0$ of  vertices  $\{q_1,q_2,q_3\}$ in $\R^2$  and the  maps $\psi_i:\R^2\to\R^2$,  $i=1,2,3$,  defined by
\[\psi_i(x):=q_i+\half(x-q_i).\]
Iterating the $\psi_i$'s, we get the set
\[\displaystyle{V^\infty=\cup_{n=0}^\infty V^n}\]
where  each $V^n$ is given  by the union of the images of    $V^0$ under the action of the maps  $\psi_w=\psi_{w_1} \circ\dots\circ\psi_{w_n}$  with $w=(w_1,\dots, w_n)$,  $w_i\in\{1,2,3\}$, a word     of length $|w|=n$.
Then the Sierpinski gasket  $\cS$ is the closure of $V^\infty$ (with respect to the Euclidean topology) and it is the unique non empty compact set $F$ which satisfies
\[
F=\cup_{i=1}^{3}\psi_i(F).
\]
For any $n$, we can identify $V^n$ with the graph $(V^n,\sim_n)$, where   $\sim_n$ is the following relation on $V^n$: for $x,y\in V^n$, $x\sim_n y$ if and only if the segment connecting $x$ and $y$ is the image of a side of the starting simplex  under the action of some $\psi_w$ with $|w|=n$. If $x,y\in V^n$ and  $x\sim_n y$, then we will say that  $x$, $y$ are \emph{adjacent} in $V^n$.\\
Given a  set $K\subset V^n\setminus V^0$, we define     the  boundary and the closure   of $K$ by
\begin{align*}
   &\pd K=\{ y\in V^n\setminus K:\, \exists x\in K  \,\text{s.t.}\, y\sim_n x\},  \\
    &\overline K=K\cup\partial K.
\end{align*}
The distance between two adjacent vertices $x,y\in V^n$ is
\[d_n(x,y)=\frac{1}{2^n}=:\d_n,\]
 while the distance between   $x,y\in V^n$ is the {\it vertex distance}
\begin{equation}\label{geodistD}
d_{n}(x,y):=\min\{d_n(x_0,x_1)+d_n(x_1,x_2)+\dots+d_n(x_{N-1},x_N) \}
\end{equation}
where the minimum  is   over all the finite path $\{x_0=x,x_1,\dots, x_N=y\}$  with $x_i\sim_n x_{i+1}$, $i=0,\dots,N-1$, connecting $x$ to $y$.\\
We also consider  the distance $d_{n,K}$ between   $x,y\in \overline K$  defined as in \eqref{geodistD}
where the minimum  is taken over all the finite paths restricted to stay inside $K$, i.e. $\{x=x_0,\dots, x_N=y\}$    with $x_i\sim_n x_{i+1}$ for $i=0,\dots,N-1$ and $x_i\in K$ for  $i=1,\dots,N-1$ (if there is no path in $K$ connecting $x$ to $y$ we set $ d_{n,K}(x,y)=+\infty$). Observe that in general
$d_n(x,y)\le d_{n,K}(x,y)$ for $x,y\in \overline K$.
\begin{definition}
A set $K\subset V^n\setminus V^0$ is said to be \emph{connected} if
$d_{n,K}(x,y)<+\infty$  for any $x,y\in\overline K$.
\end{definition}
For $u:V^n\to\R$  and   a non empty  subset $K\subset V^n\setminus V^0$, we define
\begin{align}
&\Lip^n(u,   K)=\max_{x,y \in \bar K,x\neq y}\frac{|u(x)-u(y)|}{d_{n,K}(x,y)},\label{slope}\\
&\Lip^n(u,\pd K)=\max_{x,y \in \pd K,x\neq y}\frac{|u(x)-u(y)|}{d_{n,K}(x,y)},\label{slope_boundary}
\end{align}
\textbf{(note that in \eqref{slope}  the maximum is taken over the set $\bar K=K\cup \partial K$).}
For $x\in V^n\setminus V^0$
\begin{align}
 &\D_\infty^n u (x)=\max_{y\sim_n x} \{u(y)-u(x)\} +\min_{y\sim_n x} \{u(y)-u(x)\},\label{inf_n}\\
 &  \cF^n (u,x)=\max_{y\sim_n x} \f{|u(x)-u(y)|}{\d_n}.\label{loclip}
\end{align}
We give  some basic properties of the previous operators.
\begin{definition}
A function $u:V^n\to\R$ is said \emph{infinity harmonic}  in $K\subset V^n\setminus V^0$ if
\[
       \D_\infty^n u(x)=0\qquad \text{for all $x\in  K$.}
\]
\end{definition}
\begin{proposition}\label{comparison}
Let $K\subset V^n\setminus V^0$ be a connected set.
\begin{itemize}
 \item[(i)] If $u,v:V^n\to \R$ satisfy
\[
\Delta_\infty^n u\ge 0,\quad \Delta_\infty^n v\le 0 \qquad \text{in $K$}
\]
and $u\le v$ on $\pd K$, then $u\le v$ in $\overline K$.
\item[(ii)] For any $g:\pd  K\to\R$, there exists a unique infinity harmonic function  $u$   in $K$ and  such that $u=g$ on $\pd K$.
Moreover $\min_{\pd K} g\le u \le \max_{\pd K} g$
and
\begin{equation}\label{harnack}
 \text{either}\quad\min_{y\sim_n x} \{u(y)-u(x)\}<0< \max_{y\sim_n x} \{u(y)-u(x)\} \quad \text{or} \quad u(y)=u(x)\quad\forall y\sim_n x.
\end{equation}
\end{itemize}
\end{proposition}
\begin{proof}
If $K$ is connected, then it   is   connected with the boundary, i.e. for any $x\in K$  there is $y\in \pd K$ such that $d_{n,K}(x,y)<\infty$. Under this assumption,
the comparison principle in (i)  is proved in  \cite[Theorem 4]{mos}. Existence
can be proved either  as in \cite[Theorem 5]{mos}  by means of a fixed point argument or via  the constructive approach given by the Lazarus
algorithm, see Section \ref{lazarus} for details. The estimate at the boundary  is again consequence of  the  comparison principle and  since
the constants are infinity harmonic. For property \eqref{harnack}, see \cite[Lemma 3]{mos}.
\end{proof}
In the next proposition we consider properties of the distance function with respect to the graph infinity Laplacian \eqref{inf_n}.
\begin{proposition}\label{distance}
Let $K\subset V^n\setminus V^0$ be connected and $x_0  \in \pd K$. Then the function  $u(x)= d_{n,K}(x_0,x)$
(respectively, $v(x)=-d_{n,K}(x_0,x)$)
satisfies $ \Delta_\infty^n u\le 0$ (respectively, $ \Delta_\infty^n v\ge 0$) in $K$.
\end{proposition}
\begin{proof}
Consider the function $u(x)=d_{n,K}(x_0,x)$.
Given $x\in  K $, then the minimum of $u(y)-u(x)$, $y\sim_n x$, will be a negative one, along an  adjacent point
$\bar y$ which is contained in the shortest path from $x$ to $x_0$. Since the path from $\overline y$ to $x_0$ will be one vertex shorter, then
$\min_{y\sim_n x}\{u(y)-u(x)\}=u(\overline y)-u(x)=-\d_n$.\\
If $u$ does   reach the maximum in $K$  at  $x$,  then $\max_{y\sim_n x}\{u(y)-u(x)\}\le 0$; otherwise  the maximum of $u(y)-u(x)$, $y\sim_n x$, will be positive attained at a   point $\underline y$   such that $x$   is contained in the shortest path from $\underline y$  to $x_0$ and
$\max_{y\sim_n x}\{u(y)-u(x)\}=u(\underline y)-u(x)=\d_n$.
 Hence
 \[\D_\infty^n u(x)=\max_{y\sim_n x}\{u(y)-u(x)\}+\min_{y\sim_n x}\{u(y)-u(x)\}\le \d_n-\d_n=0. \]
In a similar way it is possible to prove that $v(x)=-d_{n,K}(x_0,x)$
satisfies $ \Delta_\infty^n v\ge 0$ in $K$.
\end{proof}
In the next proposition we prove  that a function $u$ is linear along any minimal path joining   two points which realizes the maximum of the slope at the boundary, see  \eqref{slope_boundary}. This is a crucial property that will be exploited  in Section \ref{lazarus} to explicitly compute an infinity harmonic function on $V^n$.
\begin{proposition}\label{linear}
Let $K\subset V^n\setminus V^0$ be connected, $u:K\to\R$ be such that $\Lip^n(u,K)=\Lip^n(u, \pd  K)$
and $x,y\in  \pd K$   such that
\[\Lip^n(u, \pd K)=\frac{|u(x)-u(y)|}{d_{n,K}(x,y)}.\]
If  $\g=\{x=x_0,x_1,\dots, x_N=y\}$ is a minimal path for $d_{n,K}$ joining  $x$ to $y$,
then $u$ is linear  along $\g$, i.e.
\[u(x_i)=\frac{d_{n,K}(x_i,x)u(y)+d_{n,K}(x_i,y)u(x)}{d_{n,K}(x,y)},\qquad i=0,\dots,N\]
\end{proposition}
\begin{proof}
Set $L=\Lip^n(u,  K)$.
If $u$ is not linear along $\g$ and since $\Lip^n(u,\pd K)= L$ there exists an index $i\in \{1,\dots,N\}$ such
that
\begin{equation}
    |u(x_i)-u(x_{i-1})|<L d_{n,K}(x_i, x_{i-1})\label{linear1}.
\end{equation}
Then, by $\Lip^n(u,K)= L$  and \eqref{linear1}
\begin{align*}
    Ld_{n,K}(y,x)&=|u(y)-u(x_{N-1})+ \dots +u(x_{i+1})-u(x_i)+u(x_i)-u(x_{i-1})+\dots+ u(x_1)-u(x)|\\
                      &< L\big[d_{n,K}(y, x_{N-1})+\dots+L d_{n,K}(x_i, x_{i-1})+\dots+d_{n,K}(x_{1}, x)\big]=Ld_{n,K}(y,x)
\end{align*}
and therefore a contradiction.
\end{proof}
The next proposition connects the functional $\cF^n(u,x)$, which can be interpreted as the  Lipschitz constant of $u$ at $x$,
with the Lipschitz constant $\Lip^n(u,K)$ in  $K$.
\begin{proposition}\label{Lip_F}
For $u:V^n \to\R$ and for any connected set $K\subset V^n\setminus V^0$,
\begin{equation}\label{lip_equiv_func}
\Lip^n(u,  K) =\max_{x  \in K} \cF^n(u,x)
\end{equation}
 \end{proposition}
 \begin{proof}
 It is clear that $\cF^n(u,x)\le \Lip^n(u,K)$ for any $x\in K$, hence $\max_{x  \in K} \cF^n(u,x)\le \Lip^n(u,  K)$.
Let $x$, $y\in \overline K$ be such that $\Lip^n(u,K)=|u(x)-u(y)|/d_{n,K}(x,y)$.
Consider a minimal path  $\g=\{x=x_0,x_1,\dots, x_N=y\}$ from $x$ to $y$. Hence
 \begin{align*}
    \frac{|u(x)-u(y)|}{d_{n,K }(x,y)}\le \sum_{i=0}^{N-1}\frac{ |u(x_{i+1})-u(x_i)|}{d_{n,K}(x,y)}\le \f{\d_n}{d_{n,K}(x,y)}( \sum_{i=1}^{N-1}  \cF^n(u,x_i)+\cF^n(u,x_1))
 \end{align*}
 Since $\g$ is composed by $N$ arcs, then $d_{n,K}(x,y)=N\frac{1}{2^n}=N\d_n$, hence
 \[ \frac{|u(x)-u(y)|}{d_{n,K}(x,y)}\le \frac{1}{ N}\sum_{i=0}^{N-1} \max_{x  \in K} \cF^n(u,x)=\max_{x  \in K} \cF^n(u,x)\]
and therefore \eqref{lip_equiv_func}.
 \end{proof}

\section{The Absolutely Minimizing Lipschitz Extension problem in $V^n$}\label{sec3}
In this section we fix $n\in \N$ and we consider the Lipschitz extension  of a function $g:V^0\to\R$ to $V^n$. The following result is the analogous  of the the classical      Whitney-McShane solution  to the  Lipschitz extension problem.
\begin{proposition}
Given   $K\subset V^n\setminus V^0$    connected and  a function $g:\pd K\to \R$, set $L_0=\Lip^n(g,\pd K)$ and
define
\[\cM_*(x)=\max_{y\in \pd K}\{g(y)-L_0d_{n,K }(x,y)\}, \quad\cM^*(x)=\min_{y\in \pd K}\{g(y)+L_0d_{n,K }(x,y)\}.\]
Then $\Lip^n(\cM^*,K)=\Lip^n(\cM_*,K)=L_0$ and for any $u:V^n\to\R$ such that $u=g$ on $\pd K$ and $\Lip(u,K)=L_0$, we have
\begin{equation}\label{MW}
   \cM_*(x)\le u(x)\le \cM^*(x),\quad x\in K.
\end{equation}
\end{proposition}
\begin{proof}
 By the very definition of $\Lip(u,K)$, for any $x\in K$, $y\in \pd K$
\begin{align*}
    g(y)-L_0d_{n,K}(x,y)=u(y)-L_0d_{n,K}(x,y)\le u(x)\\
    u(x)\le u(y)+L_0d_{n,K}(x,y)=g(y)+L_0d_{n,K}(x,y)
\end{align*}
and therefore \eqref{MW}. Let us prove that $\Lip( \cM_*(x), V^n)=L_0$. Given $x_1, x_2\in K$, let $y_1\in \pd K $ be such that $\cM_*(x_1)=g(y_1)-L_0d_n(x_1,y_1)$. Hence,
\[\cM_*(x_1)-\cM_*(x_2)\le g(y_1)-L_0d_{n,K}(x_1,y_1)-g(y_1)+L_0d_{n,K}(x_2,y_1)\le L_0d_{n,K}(x_1,x_2).\]
The proof that $\cM_*(x_1)-\cM_*(x_2)\ge -L_0d_{n,K}(x_1,x_2)$ is similar.
\end{proof}
Our approach to the Absolutely Minimizing Lipschitz Extension problem on $V^n$ is based on the following properties/observations:
\begin{enumerate}
  \item If   $\D_\infty^n u (x)=0$, then $t=u(x)$ minimizes the functional $I(t)=\max_{y\sim_n x} \f{|t-u(y)|}{\d_n}$ giving the Lipschitz
           constant of $u$ at $x$ (see   \cite[Theorem 5]{o}  and  \cite{lg}).
  \item $\Lip^n(u,K)=\max_{x  \in K} \cF^n(u,x)$ (see Prop. \ref{Lip_F}).
  \item If $u$ solves   $\D_\infty^n u (x)=0$ in  $K\subset V^n\setminus V^0$    connected,  then $u$ is linear along a    minimal path for $d_{n,K}$ joining two points which realize the maximal slope $\Lip^n(u,\pd K)$ at the boundary (\cite{g,l}).
            The same property holds also if  $\Lip^n(u,K)=\Lip^n(u, \pd  K)$ (see Prop. \ref{linear}).
 \item   The functions  $\cM_*(x)$,  $ \cM^*(x)$, defined by means of the distance function $d_{n,K}$,  give the minimal and
  maximal solution to the   Lipschitz extension problem  on $K$. Moreover, for $x_0\in \pd K$,  $d_{n,K}(x_0, \cdot)$ and $-d_{n,K}(x_0, \cdot)$
  are a supersolution and a a subsolution of $\Delta_\infty^n u=0$ in $K$ (Prop. \ref{distance})
\end{enumerate}
Indeed it is clear that all the previous concepts are strictly related and, in analogy with the Euclidean case, we introduce the following definitions
    \begin{definition}\label{maindef}\mbox{}
    \begin{itemize}
    \item[(i)]  A function $u:V^n\to\R$ is said an \emph{absolute minimizer for the functional $\Lip^n$} on $V^n$ if for any  connected set $K\subset V^n\setminus V^0$  and for any $v:V^n\to\R$ such that $u=v$ on $\pd K$, then $\Lip^n(u,  K)\le \Lip^n(v,  K)$. We denote by $\AML(V^n)$ the  set of  the  absolute minimizers for $\Lip^n$ in $V^n$.
    \item[(ii)] A function $u:V^n\to\R$ satisfies the Comparison with Cones  property (noted CC property) in $V^n$ if for any connected set $K\subset V^n\setminus V^0 $, for any $x_0\in  \pd  K$, $\l\ge 0$ and  $\a\in\R$
           \begin{align*}
            u\le \l d_{n,K}(x_0,\cdot)+\a\quad \text{on $\partial K$ implies}\quad u\le \l d_{n,K}(x_0,\cdot)+\a \quad\text{on $K$,}\\
            u\ge -\l d_{n,K}(x_0,\cdot)+\a\quad \text{on $\partial K$ implies}\quad u\ge - \l d_{n,K}(x_0,\cdot)+\a \quad\text{on $K$.}
           \end{align*}
      \item[(iii)]  A function  $u$ is said an \emph{absolute minimizer for the functional $\cF^n$} on $V^n$ if for any   $x \in V^n\setminus V^0$ and for any $v:V^n\to\R$ such that $v(y)=u(y)$ for all $y\sim_n x$, then $\cF^n(u,x)\le \cF^n(v,x)$.
            We denote by    $\AM(V^n)$ the set of  the  absolute minimizer for $\cF^n$ in $V^n$.
   \end{itemize}
\end{definition}

We have the following result.
\begin{proposition}\label{equivalence}\mbox{}
The following properties are equivalent
\begin{itemize}
  \item[(i)]    $u\in \AML(V^n)$;
  \item[(ii)]   $u$ satisfies the CC property in $V^n$;
   \item[(iii)]  $u$ is infinity harmonic in $V^n\setminus V^0$;
  \item[(iv)]   $u\in \AM(V^n)$.
\end{itemize}
\end{proposition}
\begin{proof}
To show that (i) implies (ii),  assume by contradiction that there exists  a connected set $K\subset V^n\setminus V^0$, $x_0\in \pd K$, $\a\in\R$ and $\l\ge 0$
such that
\[
    u(x)\le \a+\l d_{n,K}(x_0,x)\qquad\forall x\in \pd K
\]
and that the set $W=\{y\in K:\,u(y)>\a+\l d_{n,K}(x_0,y)\}$ is not empty (we   assume that $W$ is connected otherwise we consider a connected component of $W$).
Observe that  $\pd W\subset \overline K$ and
\begin{equation}\label{nv1}
 u(x)\le \a+\l d_{n,K}(x_0,x)\qquad\forall x\in \pd W.
\end{equation}
 Let $y_1,y_2\in \pd W$ such that, defined  $L:=\Lip^n(u, \pd W)$, then
\[
L=\frac{u(y_2)-u(y_1)}{d_{n,W}(y_1,y_2)}.
\]
Let $\g$ be a minimal path for $d_{n,W}(y_1,y_2)$, $y_0\in \g\cap W$ (this point exists by the  definition of $\Lip^n(u, \pd W)$) and $z_0\in \pd W$
such that $d_{n,K}(y_0,x_0)=d_{n,K}(y_0,z_0)+d_{n,K}(z_0,x_0)=d_{n,W}(y_0,z_0)+d_{n,K}(z_0,x_0)$. Hence by \eqref{nv1}
\begin{equation}\label{nv2}
\begin{split}
    &u(x)\le \b+\l d_{n,W}(z_0,x)\qquad\forall x\in \pd W,\\
    &u(y_0)> \b+\l d_{n,W}(z_0,y_0),
    \end{split}
\end{equation}
where $\b=\a+\l  d_{n,K}(z_0,x_0)$. Since $u\in \AML(V^n)$,
by Prop. \ref{linear} $u$ is linear along $\g$ and therefore  $u(y_2)-u(y_0)=Ld_{n,W}(y_0,y_2)$. Hence by \eqref{nv2}
\begin{align*}
&\b+\l  d_{n,W}(z_0,y_2)\ge   u(y_2) = u(y_0)+L d_{n,W}(y_0,y_2)\\
&> \b+\l d_{n,W}(z_0,y_0)+L d_{n,W}(y_0,y_2)\\
&\ge\b+\l d_{n,W}(z_0,y_2)+(L-\l)d_{n,W}(y_0,y_2).
\end{align*}
Since $d_{n,W}(y_0,y_2)>0$, then $L< \l$. On the other hand, by \eqref{nv2} and $\Lip^n(u,\pd W)=\Lip^n(u,W)$,  we have $u(z_0)\le \b$ and
\[u(z_0)+Ld_{n,W}(y_0,z_0)\ge u(y_0)> \b+\l d_{n,W}(y_0,z_0)\]
and therefore $\l< L$, hence a contradiction to $W$ not empty.\\
To prove that (ii) implies (i), assume now that $u$ satisfies (ii) and set $L=\Lip^n(u,\pd K)$. Fix $x\in K$, then  by the CC property
\begin{equation}\label{e00}
  u(z)-Ld_{n,K}(x,z)\le  u(x) \le u(z)+Ld_{n,K}(x,z)
\end{equation}
for all   $z\in \pd K$. Set $J_i$ a connected component of $K\setminus \{x\}$, then $\pd J_i\subset \pd K\cup \{x\}$ and by \eqref{e00}
 $\Lip^n(u,\pd J_i)\le L$.   Again  by the CC property we have
\[
u(x)-Ld_{n,J_i}(x,y)\le  u(y) \le u(x)+Ld_{n,J_i}(x,y), \qquad \text{for all   $y\in J_i$.}
\]
Since $d_{n,J_i}(x,y)=d_{n,K}(x,y)$  we obtain $|u(x)-u(y)|\le Ld_{n,K}(x,y).$

As $K$ is connected, then for any $z_1\in J_1$ and  for any $z_2\in J_2$ (where $J_1$ and $J_2$ denote two different connected component of $K\setminus \{x\}$) we have $$d_{n,K}(z_1,z_2)= d_{n,J_1}(z_1,x)+d_{n,J_2}(x,z_2)$$
and therefore,
since $x$ is  arbitrary in $K$, $\Lip^n(u,K)=L$.
We prove that   (iii) implies  (ii).
If $u$ is infinity harmonic in $K$, $x_0\in \pd K$,
$\a\in\R$ and $\l\ge 0$ are such that
\[u(y)\le \a +\l d_{n,K}(x_0,y),\qquad\forall y\in \pd K ,\]
then by Prop. \ref{comparison} and Prop.\ref{distance}, it immediately follows that
\[u(x)\le \a +\l d_{n,K}(x_0,x),\qquad\forall x\in  K.\]
Similarly for the other relation  in  the definition of the CC property.\\
To show that (ii) implies (iii), assume by contradiction that there exists $x\in V^n\setminus V^0$ such that
\begin{equation}\label{contr1}
\D_{\infty}^n u(x)=2\eta>0.
\end{equation} Consider the  set  $K=\{x\}$.  Hence $\pd K=\{y_i\}_{i=1}^4$ where $y_i$ are the four  points adjacent to $x$ in $V^n$.
Let $y_1$, $y_2\in \pd K$ be  the points where $u$ attains its maximum, respectively minimum, on $\pd K$ and $\l\ge 0$ be
such that $u(y_1)-u(y_2)=\l d_{n,K}(y_1,y_2)=2\l\d_n$.
Hence
\[
u(y_j)\ge u(y_2)=u(y_1)- \l d_{n,K}(y_1,y_2)=u(y_1)-\l d_{n,K}(y_1,y_j),\quad j=1,2,3,4,
\]
and therefore by the CC property in $K$
\begin{equation}\label{contr2}
  u(x)\ge u(y_1)-\l d_{n,K}(y_1,x)=u(y_1)-\l \d_n.
\end{equation}
By \eqref{contr1}, we have
\begin{align*}
    u(x)+\eta&=\frac{1}{2} \left( \min_{  i=1,\dots,4 }\{u(y_i)\}+\max_{ i=1,\dots,4 }\{u(y_i)\}\right)=\frac{u(y_1)+u(y_2)}{2}\\
    &=\frac{u(y_1)}{2}+\frac{u(y_1)-\l d_{n,K}(y_1,y_2)}{2}=u(y_1)-\l \d_n=u(y_1)-\l d_{n,K}(y_1,x)
\end{align*}
It follows that $u(x)+\eta = u(y_1)-\l d_{n,K}(y_1,x)$ and therefore a contradiction to \eqref{contr2}.\\
The equivalence between (iii) and (iv) is proved in \cite[Theorem 5]{o} observing that   $\Delta_\infty^n u(x)=0$ if and only if
 $u(x)$ is such that
\[
\cF^n(u,x)=\min_{t\in\R} \left\{ \max_{y\sim_n x} \f{|t-u(y)|}{\d_n}\right\},
\]
i.e. $u(x)$ minimizes the functional $I(t)=\max_{y\sim_n x} {|t-u(y)|}/{\d_n}$ .

\end{proof}
\begin{corollary}\label{uniqueness}
Given $g:V^0\to\R$ and defined $L_0=\max_{i,j=1,2,3}\{|g(q_i)-g(q_j)|\}$, where $\{q_1,q_2,q_3\}$ are the vertices of $V^0$,
then for any $n\in \N$  there exists a unique $u^n\in  \AML(V^n)$ such that $u=g$ on $V^0$.
\end{corollary}
\begin{proof}
The previous result  is an immediate consequence of Prop. \ref{comparison}
and Prop. \ref{equivalence}.
\end{proof}

\begin{remark}
The definition of   Lipschitz constant at the boundary,  which
is computed considering paths staying inside the domain, see \eqref{slope_boundary},
rules out  the pathological example considered in \cite[Example 2.2]{j} where a    function defined on the boundary
does not have an absolutely minimizing extension. \\
Note that  the theory developed in this section applies to a generic graph.
Let  $(X,d)$ be the metric space with $X=\{x,y,z\}$
and $d(x,y)=3/2$, $d(x,z)=d(y,z)=1$. Consider $K=\{z\}$, hence $\pd K=\{x,y\}$,  and $f :\pd K\to \R$ defined
by $f(x)=0$, $f(y)=1$.  Hence  $\Lip(f,\pd K)= 1/2$ (consider the path inside $K$ given by $x_0=y$,
$x_1=z$, $x_2=x$). If $u$ is infinity harmonic  in $K$, then $u(z)=\max\{u(x),u(y)\}/2+ \min\{u(x),u(y)\}/2=1/2$
and $\Lip(u,K)=\Lip(u,\pd K)=1/2$.\\
If   the Lipschitz constant of $f$  at the boundary is computed as $|f(y)-f(x)|/d(x,y)=2/3$,
then  it is shown in \cite{j} that an absolutely minimizing Lipschitz extension of $f$ to $X$ does not exist.
\end{remark}


\section{The Absolutely Minimizing Lipschitz Extension problem in    $\cS$}\label{sec4}
Following \cite{acj,j,js}, we introduce absolute minimizing Lipschitz extensions on $\cS$. We consider
the  length  space $(\cS,d)$    where $d$ is the path distance defined by
 \[
   d(x,y):=\inf\left\{\ell(\g):\,\text{$\g$ is a path joining $x$ to $y$ }\right\}
 \]
with $\ell(\g)$   the length of $\g$ (see \cite{hs}). We consider as boundary of $\cS$, as for all the sets $V^n$,  the initial set
$V^0=\{q_1,q_2,q_3\}$ and we assume that a function $g:V^0\to\R$ is given.\\
Given   $A\subset \cS$ and  $f:A\to\R$, we define
 the Lipschitz constant of $f$ on $A$ to be
 \[
  \Lip(f,A):=\sup_{x,y\in \overline A, x\neq y}\frac{|f(y)-f(x)|}{d(x,y)}
 \]
\begin{definition}\hfill
 \begin{itemize}
    \item  A continuous function $u:\cS\to\R$ is said an \emph{absolute minimizer for the functional $\Lip$} on $\cS$ if for any  proper, open, connected  set $A\subset \cS\setminus V^0$  and for any $v:\cS\to\R$ such that $u=v$ on $\pd A$, then $\Lip(u,  A)\le \Lip(v,  A)$.
        We denote by $\AML(\cS)$ the   set of  the  absolute minimizer for $\Lip$ in $\cS$.
    \item A function $u:\cS\to\R$ satisfies the Comparison with Cones  property (noted CC property)  if for any  proper, any  proper, open, connected  set $A\subset \cS\setminus V^0,$ for any $x_0\in  \cS\setminus A$, $\l\ge 0$ and  $\a\in\R$
           \begin{align*}
            u\le \l d (x_0,\cdot)+\a\quad \text{on $\partial A$ implies}\quad u\le \l d (x_0,\cdot)+\a \quad\text{on $A$}\\
            u\ge -\l d (x_0,\cdot)+\a\quad \text{on $\partial A$ implies}\quad u\ge - \l d (x_0,\cdot)+\a \quad\text{on $A$}
           \end{align*}
   \end{itemize}
\end{definition}
In the following proposition we summarize the   results in \cite{j,js} concerning the existence of AMLE in metric spaces which in particular applies to $(\cS,d)$.
\begin{proposition} \label{ext_Sierpinski}\hfill
\begin{itemize}
\item[(i)] A function $u$ is of class $\AML(\cS)$ if and and only if satisfies the Comparison with Cones property.
\item[(ii)] For any given $g:V^0\to\R$, there exists a unique  function $u\in \AML(\cS)$ such that $u=g$ on $V^0$.
\end{itemize}
\end{proposition}
 \begin{proof}
For the proof of (i), see \cite[Proposition 4.1]{js},  for the one of (ii), we refer to   \cite[Theorem 4.3]{j} for existence and
to \cite[Theorem 1.4]{pssw} for the uniqueness.
\end{proof}
\begin{theorem}\label{thm_pre}
Given $g:V^0\to \R$,   let $u^n$ be the $\AML(V^n)$  of $g$ to $V^n$. Then
\[
   \lim_{n\to \infty}u^n(x)=u(x),\qquad\text{uniformly in $\cS$},
\]
i.e. $\lim_{n\to\infty}\sup_{x\in V^n}|u^n(x)-u(x)|=0$, where $u$ is  the  $\AML(\cS)$  of  $g$ to $\cS$.
\end{theorem}
\begin{proof}
Since the sequence $\{u^n\}_{n\in\N}$ is uniformly bounded and equi-Lipschitz continuous, there exists a function  $u:\cS\to\R$ such that, up to a subsequence,  $\lim_{n\to \infty}u^n(x)=u(x)$ and  the convergence is also uniform. Moreover $u$ is Lipschitz continuous with Lipschitz constant $L$.
We prove that $u$ satisfies the CC property on $\cS$. Given a proper, open set  any  proper, open, connected  set $A\subset \cS\setminus V^0,$   let $x_0\in \cS\setminus A$, $\a\in\R$, $\l> 0$ (if $\l=0$ the argument is similar) be such that
\[
    u(y)\le \a+\l d(x_0,y),\qquad \forall y\in \pd A,
\]
and assume by contradiction that the set $W=\{y\in A:\,u(y)>\a+\l d(x_0,y) \}$ is not empty
(we   assume that $W$ is connected otherwise we consider a connected component of $W$). Observe that
\[
    u(y)=\a+\l d(x_0,y)\qquad  \forall y\in \pd W.
\]
Let $y_0\in W$ and $\eta>0$ be such that
\begin{equation}\label{conv15}
u(y_0)\ge\a+\l d(x_0,y_0)+ 4\eta.
\end{equation}
Defined  $W_n=W\cap V^n, $  let $\varepsilon=\frac{\eta}{L+\lambda} .$  As $W$ is open, there exists  an integer $n_0$
  and $y_0^*\in W_{n_0} $  such that $d(y_0, y_0^*)<\varepsilon.$
 Since
 $W_n\subset W_{n+1}$,  then $y_0^*\in W_{n} $  for any $n\geq n_0.$

Therefore by \eqref{conv15} and the continuity of $u$, we have
\[
  u(y_0^*)\ge  u(y_0) -L \varepsilon \ge \a+\l d(x_0,y_0^*)+3 \eta.
\]

%
%

Similarly, there exists a point $x^*_0 \in V^{n_1}$ such that  $d(x_0, x_0^*)<\frac{\eta}{2\lambda}:$
then
$$u(y^*_0)\ge \a+\l d(x_0,y_0^*)+3 \eta
  \ge \a +\l (d(y^*_0,x_0^*)+ d(x_0,x_0^*)-2d(x_0,x_0^*) ) +3\eta   \ge \a +\l (d(y^*_0,x_0^*)+ d(x_0,x_0^*)) +2\eta.
 $$
Defined $\beta=\a+\l d(x_0,x_0^*),$
we have
$$u(y^*_0)\ge \beta +\l d(y^*_0,x_0^*)+2\eta.
 $$
Moreover
$$u(y) \le\b+\l d(x^*_0,y)\qquad \forall y\in A \setminus W.$$
Let $\g_n\subset V^n$ be a path $\{x^*_0= x^n_0,x^n_1,\dots, x^n_{N}=y^*_0\}$  with $x^n_i\sim_n x^n_{i+1}$, $i=0,\dots,N-1$, connecting $x^*_0$ to $y^*_0$ such that
$d_n(x^*_0,y^*_0)=\ell(\g_n)$. Since $d_n(x^*_0,y^*_0)\to d(x^*_0,y^*_0)$ for $n\to \infty$
(see for example \cite[Corollary 5.1]{ccm}, \cite{hs}), then
 \[
 d(x^*_0,y^*_0)\ge \ell(\g_n)-\frac{\eta}{\l}
 \]
 for $n$ sufficiently large, where $\ell(\g_n)$ is the length of $\g_n$. Denoted by $z_n$ the first point $x^n_i\in  \g_n$, $i=1,\dots, N-1$,
such that $x^n_{i}\not\in W$,
$$d_n(x^*_0,y^*_0)= d_n(x^*_0,z_n)+d_n(z_n,y^*_0).$$
and by (4.4)
$$d(x^*_0,y^*_0)\ge d_n(x^*_0,z_n)+d_n(z_n,y^*_0)-\frac{\eta}{\l}.$$
Set $\b_n=\b+\l d_n(x^*_0,z_n)$, then
\begin{align}
&u(y) \le\b+\l d(x^*_0,y)\le \b+\l d_n(x^*_0,y)\le \b+\l d_n(x^*_0,z_n)+\l d_n(z_n,y)=\b_n+\l d_n(z_n,y) \qquad  \forall y\in A \setminus W, \label{conv18}\\
&u(y^*_0)\ge \beta+\l d(x^*_0,y^*_0)+2\eta\ge \b+\l d_n(x^*_0,z_n)+\l d_n (z_n,y^*_0)-\eta+2\eta =\b_n+\l d_n(z_n,y^*_0)+\eta.\label{conv19}
\end{align}
For any $y\in \pd W_n$ we have in particular that $y\in A \setminus W$ and
by the uniform convergence of $u^n$ to $u$ and  \eqref{conv18},    there exists $\e_n\to 0$ for $n\to\infty$ such that
\[
    u^n(y)\le \b_n+\e_n+\l d_{n, W^n}(z_n,y)\qquad \forall y\in \pd W_n.
\]
Then  by the CC property for $u^n$, we get (in particular)
\begin{equation}\label{conv6}
    u^n(y^*_0)\le \b_n+\e_n+\l d_{n, W^n}(z_n,y^*_0)\qquad.
\end{equation}
Passing to the limit for $n\to\infty$  in \eqref{conv6} we get a contradiction to \eqref{conv19}.
In fact $$ u(y^*_0)\le \liminf(\b_n+2 \e_n+\l d_{n, W^n}(z_n,y^*_0)  )$$
and by definition of $z_n,$  we have $d_{n, W^n}(z_n,y^*_0) \leq  d_{n}(z_n,y^*_0)$
so $$ u(y^*_0)\le \liminf(\b_n+2 \e_n+\l d_{n}(z_n,y^*_0)  )=\liminf(\b_n+\l d_{n}(z_n,y^*_0)  )$$
while from \eqref{conv19} we can deduce
$$ u(y^*_0)\ge \limsup(\b_n+\l d_{n}(z_n,y^*_0)  )+\eta$$
Arguing in a similar way for the other relation, we conclude that $u$  satisfies the CC property on $\cS$ and
therefore $u\in\AML(\cS)$.  By the uniqueness of $u$, see Prop. \ref{ext_Sierpinski}, we get that all the
sequence $u^n$ converges uniformly to $u$.
\end{proof}
\begin{definition}
We say that a function $u:\cS\to\R$ is infinity harmonic if it is the limit of infinity harmonic functions $u^n:V^n\to \R$
such that $u^n=u$ on $V^0$.
\end{definition}
By Theorem \ref{thm_pre}, we have
\begin{corollary}
The following properties are equivalent
\begin{itemize}
\item[(i)] $u\in\AML(\cS)$;
\item[(ii)]  $u$ satisfies  the Comparison with Cones property;
\item[(iii)] $u$ is infinity harmonic.
\end{itemize}
\end{corollary}

\begin{remark}
As for the Laplacian on the Sierpinski gasket (see \cite{s}), one would be tempted to say that, given $f:V^{n-1} \to\R$, $\tilde f$ is an infinity harmonic
extension of $f$ to $V^n$ if $\tilde f\in \AML(V^n)$  and $\tilde f=f$ on $V^{n-1}$. But a function $\tilde f$ with such property
could   not exist. This can be seen with the following   example. Let $g:V^0\to\R$
such that $g(q_1)=0$, $g(q_2)= e\in[0,1/7]$ and $g(q_3)=1$ and denote  by $q_{ij}=q_{ji}$, $i,j=1,2,3$ and $i\neq j$, the point in $V^1\setminus V^0$ on the segment of vertices $q_i$ and $q_j$.  By the Lazarus algorithm  (see Section \ref{lazarus})  the unique solution of $\D_\infty^1 u(x)=0$ in $V^1\setminus V^0$ and $u=g$ on $V^0$
 is given by $u^1(q_{12})=(1+e)/4$ and $u^1(q_{23})=(1+e)/2$, $u^1(q_{13})=1/2$.\\
If we consider the problem  $\D_\infty^2 u(x)=0$ in $V^2\setminus V^0$ and $u=g$ on $V^0$, by applying again the Lazarus algorithm we find $u^2(q_{12})=(3+4e)/12$ and therefore $u^2(q_{12})\neq u^1(q_{12})$ if $e\neq 0$. Hence the function $u^2$ does not satisfy $\D_\infty^1 u^2(x)=0$ in $V^1$, since  otherwise $u^2\equiv u^1$ on $V^1$. Note that the set $K=V^1\setminus V^0$ is not connected as a subset of $V^2$ and $\pd K=V^2\setminus V^0$, hence    the
values of $u^2$  on $V^2$ are used to compute $\D_\infty^2 u^2(x)=0$ at  $x\in V^1$.\\
 This is a main  difference with the case of the Laplacian where the harmonic extension $\tilde f$ of a function $f$ to $V^n$ still satisfies the Laplace equation on $V^{n-1}$. The main consequence of this observation is that we cannot define an infinity harmonic function on $\cS$ as a continuous function whose restriction  to $V^n$ is infinity harmonic for all $n\in \N$.
\end{remark}
The following property   can be useful in order to characterize the limit of the sequence $u^n$.
\begin{proposition}
For $u:\cS \to\R$,  the functional
\[
\cF^n(u,V^n)=:\max_{x\in V^n\setminus V^0}\cF^n(u,x)
\]
 is increasing in $n\in \N$.
\end{proposition}
\begin{proof}
Let $x\in V^n\setminus V^0$  be  such that  $\cF^n(u,V^n)= \cF^n(u,x)$,   $y\in V^n$   such that $\cF^n(u,x)=|u(x)-u(y)|/\d_n$
and $z$ the point in $V^{n+1}$ in between $x$ and $y$. Then
\begin{align*}
    \frac{|u(x)-u(y)|}{\d_n} \le \half \frac{|u(x)-u(z)|}{\d_{n+1}}+\half \frac{|u(z)-u(y)|}{\d_{n+1}}\\
    \le \half \cF^{n+1}(u,z)+\half \cF^{n+1}(u,z)\le \cF^{n+1}(u,V^{n+1})
\end{align*}
hence $\cF^n(u,V^n)\le \cF^{n+1}(u,V^{n+1})$.
\end{proof}

\section{Some complements to the AMLE problem in $V^n$ }\label{sec5}
We discuss in this section two further  properties of   the AMLE problem in $V^n$:
\begin{itemize}
  \item an algorithm to compute explicitly an infinity harmonic function on $V^n$;
  \item the relation between $p$-harmonic   and infinity harmonic functions on $V^n$.
\end{itemize}
\subsection{A constructive approach: The Lazurus algorithm}\label{lazarus}
We describe   an algorithm introduced  in \cite{l} and extensively studied  in \cite{g} which allows to compute  an infinity harmonic function  on $V^n$. \\
Observing that the set of  infinity harmonic functions
is invariant by   addition and multiplication for constants,    if $u$ is infinity harmonic, then
\begin{equation}\label{laz1}
v(x)=\f{u(x)-\min_{ V^0}\{ u\}}{\max_{ V^0}\{ u\}-\min_{ V^0}\{ u\}}
\end{equation}
is also infinity harmonic. Moreover $v$ assumes   the boundary values $0$, $e$, $1$ for  some $e\in [0,1]$. Since the
values at the boundary determine univocally an infinity harmonic function,   to compute $u$ is sufficient to
consider the case of the boundary values $0$, $e$, $1$ and then   to inverte the affine transformation \eqref{laz1}. It is possible to further
reduce the computation  by considering $e\in [0,1/2]$. In fact if $e\in [1/2,1]$, then
$w(x)=1-v(x)$  is infinity harmonic and  assume the boundary values $0$, $1-e\in [0,1/2]$, $1$ (obviously the order of the vertices of $V^0$
where these values are assumed is not relevant). Computed $w$,   we have $v(x)=1-w(x)$. \par
We start with  the boundary values $v(q_1)=0$ $v(q_2)= e$, $v(q_3)= 1$ (see figure \ref{fig1}.(a)) and we compute the corresponding infinity harmonic function on $V^1$. We   denote  by $q_{ij}=q_{ji}$, $i,j=1,2,3$ and $i\neq j$, the point in $V^1\setminus V^0$ on the segment of vertices $q_i$ and $q_j$.
Exploiting   Prop.\ref{linear}, since
$$\Lip^1(v,\pd (V^1\setminus V^0) )=\f{v(q_3)-v(q_1)}{d_1(q_1,q_3)}=1,$$
 the function $v$ is linear
along the minimal path $\{q_1,q_{13}, q_3\}$ for $d_1(q_1,q_3)$,  hence $v(q_{13})=1/2$.\\ To compute the other values we
consider the connected set $K= \{q_{12},q_{23}\}$ with $\pd K=V^0\cup\{q_{13} \}$.
We distinguish two cases:\par
\emph{(i):} If $e\in [1/3,1/2]$, then
$$\Lip^1(v,\pd K )=\f{v(q_3)-v(q_1)}{d_{1,K}(q_1,q_3)}=\f{1}{3/2}
$$
and a minimal path for $d_{1,K}(q_1,q_3)$ is given by $\{q_1,q_{12}, q_{23},q_3\}$. Applying again Prop. \ref{linear}, we get
$v(q_{12})=1/3$, $v(q_{23})=2/3$ (see figure \ref{fig1}.(b)).\par
\emph{(ii):} If $e\in [0,1/3]$, then
$$\Lip^1(v,\pd K )=\f{v(q_3)-v(q_2)}{d_{1,K}(q_2,q_3)}=\f{1-e}{1}
$$
and a minimal path for $d_{1,K}(q_2,q_3)$  is given by $\{q_2,  q_{23},q_3\}$. Hence $v(q_{23})=(1+e)/2$. To determine $v(q_{12})$, we consider
the connected  set $J=\{q_{12}\}$ with boundary $\pd J=\{q_1, q_{13}, q_{23}, q_2\}$. Since
$$ \Lip^1(v,\pd J )=\f{v(q_{23})-v(q_1)}{d_{1,K}(q_1,q_{23})}=\f{(1+e)/2}{1}
$$
and a  minimal path for $d_{1,K}(q_1,q_{23})$  is given by $\{q_1, q_{12},  q_{23}\}$, we get $v(q_{12})=(1+e)/4$ (see figure \ref{fig1}.(c)).\\
Arguing as above it is possible (in principle) to compute infinity harmonic function on $V^n$ for any $n\in \N$  (the case   $n=2$
is detailed in \cite{g}). Alternatively it is possible to use the iterative scheme  developed in \cite{o} in the framework of numerical
approximation of infinity harmonic functions in Euclidean space. Since this scheme works for general graph, it can be applied to $V^n$.\\
It is worth noticing that a somewhat stronger result   than Theorem \ref{thm_pre} is obtained  in \cite{g}. In fact it is proved that
the restriction of an infinity harmonic function $u^n$ on $V^n$   is eventually unchanging on $V^k$ for $n$ sufficiently
larger than  $k$.  This result is based on the explicit construction of the optimal paths  for the Lazarus algorithm and the argument is rather involved.


\begin{figure}[h!]\label{fig1}
\begin{center}
\begin{tabular}{ccc}
\includegraphics[height=3.5cm]{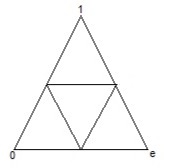}\quad &
\includegraphics[height=3.5cm]{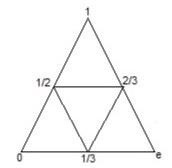}\quad&
\includegraphics[height=3.5cm]{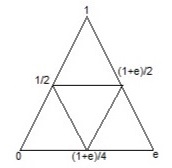}  \\
(a)& (b) & (c)
\end{tabular}
\end{center}
\caption{infinity harmonic function on $V^1$: (a) boundary condition, (b) $e\in [1/3,1/2]$, (c) $e\in [0,1/3]$}
\end{figure}

\subsection{Infinity and   $p$-harmonic function in $V^n$}
 Another important point of view of the  AMLE  theory  is  the relation between  $p$-harmonic  and infinity harmonic functions. In the Euclidean case  it can be proved that the limit of $p$-harmonic functions for $p\to \infty$ is an infinity harmonic function \cite{acj,c}, while this property may fail in general metric-measure spaces \cite{js}. We start to prove a similar result for the pre-fractal $V^n$.
Fixed $n\in\N$ and given $g:V^0\to\R$, we introduce for   $p\in [1,\infty)$      the $p$-energy functional
\begin{equation}\label{penergy}
 I^n_p(u)=\left(\sum_{x\in V^n}\sum_{y\sim_n x}\left|\f{u(y)-u(x)}{\d_n}\right|^p\right)^{\f 1 p}.
\end{equation}
A function which achieves the minimum in \eqref{penergy} is said a \emph{$p$-harmonic function} on $V^n$. The existence of a $p$-harmonic function on graphs  is studied in \cite{hso}.
\begin{proposition}\label{pconvergence}
Let $\{u^n_p\}_{p\ge 1}$ be  the family of the $p$-harmonic functions on $V^n$ such that $u^n_p=g$ on $V^0$. Then $u^n_p\to u^n$ for $p\to \infty$
where $u^n\in \AML(V^n)$ and $u^n=g$ on $V^0$.
\end{proposition}
We need a preliminary lemma,   expressing the local character of $p$-harmonic functions.
For $x\in V^n\setminus V^0$ and  $v:V^n\to\R$,  set $$I^n_p(v,x)=\left(\sum_{y\sim_n x}\left|\f{v(y) -v(x)}{\d_n}\right|^p\right)^{\f 1 p}.$$
 \begin{lemma}\label{convp0}
If $u^n_p$ is $p$-harmonic,  then for any $v:V^n\to \R$ such  that $v(y)=u^n_p(y)$ for $y\sim_n x$, we have $I^n_p(u^n_p,x)\le I^n_p(v,x)$.
\end{lemma}
\begin{proof}
If by contradiction there exists $v:V^n\to \R$ such that $v(y)=u^n_p(y)$ for $y\sim_n x$ and such that $I^n_p(v,x)\le I^n_p(u^n_p,x)-\e$ for some $\e>0$, then the function $\bar u:V^n\to\R$ defined by
$$
\bar u(y)=\left\{
            \begin{array}{ll}
              u^n_p(y), & \hbox{$y\neq x$;} \\
               v(x), & \hbox{$y=x$.}
            \end{array}
          \right.
$$
is such that $I^n_p(\bar u)< I^n_p(u^n_p)$. Hence a contradiction to the definition of $p$-harmonic function.
\end{proof}
\begin{proof}[Proof of Prop. \ref{pconvergence}]
Let $u^n_p$ be a $p$-harmonic function and $u^n\in \AML(V^n)$ such that $u^n=g$ on $V^0$. Then
\begin{align*}
    I_p(u^n_p)\le I_p(u^n)=\left(\sum_{x\in V^n\setminus V^0}\sum_{y\sim_n x}\left|\cF^n(u^n,x)\right|^p+\sum_{i=1}^3\sum_{x\sim_n q_i}  \left|\f{u^n(x) -u^n(q_i)}{\d_n}\right|^p\right)^{\f 1 p}
\end{align*}
where $\cF^n$ is defined as in \eqref{loclip} and $q_i$, $i=1,2,3$ are the vertices of $V^0$. Hence, by \eqref{lip_equiv_func}
\[
   I_p(u^n_p)\le (4 N+6)^{\f 1 p}\max_{V^n\setminus V^0}\{\cF^n(u^n,x)\}= (4N+6 )^{\f 1 p}\,\Lip^n(u^n,V^n)= (4 N+6)^{\f 1 p} L_0
\]
where $L_0=\max_{i,j=1,2,3}\{g(q_i)-g(q_j)\}$ and $N=\#(V^n\setminus V^0)$ (hence $N=(3^{n+1}-3)/2$). It follows that for any $x\in V^n\setminus V^0$ and
$y\sim_n x$, $|u^n_p(y)-u^n_p(x)|$ is bounded, uniformly in $p\in [1,\infty)$. Since $u^n_p=g$ on $V^0$, then  the functions $u^n_p$ are uniformly bounded in $p$ on $V^n$. Passing to a subsequence, we get that there exists a sequence $\{u_{p_j}\}_{j\in\N}$ which converges to a function $\bar u$ on $V^n$ such that $\bar u=g$ on $V^0$.\\
We claim that $\bar u \in \AM^n(V^n)$ (see Def. \ref{maindef}.(iii)). Let  $x\in V^n$ and   $v:V^n\to \R$ such that $v(y)=u^n_{p_j}(y)$ for $y\sim_n x$.
Then by Lemma \ref{convp0} we have
\begin{align*}
    \left(\sum_{y\sim_n x}\left|\f{u^n_{p_j}(y)-u^n_{p_j}(x)}{\d_n}\right|^p\right)^{\f 1 p}
    \le \left(\sum_{y\sim_n x}\left|\f{u^n_{p_j}(y)-v(x)}{\d_n}\right|^p\right)^{\f 1 p} \qquad \forall j\in \N.
\end{align*}
Passing to the limit for $j\to \infty$ in the previous inequality, since the $p$-norm  in $\R^N$  converges
 to the $\infty$-norm  we get
 \begin{align*}
    \max_{y\sim_n x}\left|\f{\bar u (y)-\bar u(x)}{\d_n}\right|    \le \max_{y\sim_n x}\left|\f{\bar u(y)-v(x)}{\d_n}\right|.
\end{align*}
Hence  $\cF^n(\bar u,x)\le \cF^n(v,x)$ and therefore the claim. Since there exists  a unique $u^n \in  \AM^n(V^n)$ such that $u^n=g$ on $V^0$,
 see  Prop. \ref{equivalence} and Cor. \ref{uniqueness}, it follows that $\bar u=u^n$. Moreover by the uniqueness of $u^n$, any convergent sequence of  $\{u^n_p\}_{p\ge 1}$ tends to $u^n$. Therefore we conclude that $\lim_{p\to\infty} u^n_p=u^n$.
\end{proof}

 We proved the convergence  $u^n_p\to u^n$ for $p\to \infty$ in Proposition \ref{pconvergence}  and the convergence   $u^n\to u$ for $n\to \infty$ in  Theorem \ref{thm_pre}.   A natural question is if it possible to invert the order of the limits.
 Actually we have only partial results.
  A notion of   $p$-harmonic functions on $V^n$ is studied in \cite{hps} with the aim of defining   a $p$-Laplace operator on the Sierpinski gasket in the spirit of Kigami's approach. The  discrete $p$-energy  ${\mathcal{E}}^{(n)}_p $ considered in \cite{hps} is different from  \eqref{penergy}, even if it is dominated from below and from above by $(I^n_p)^p$.
The sequence of the discrete energies  ${\mathcal{E}}^{(n)}_p $   is increasing, hence it is possible to define the $p$-energy ${\mathcal{E}}_p$ on  the Sierpinski gasket $\cS$ as the limit
    \begin{equation}\label{E}
    {\mathcal{E}}_{p} (u)= \lim_{n\to\infty}{\mathcal{E}}^{(n)}_p (u).
    \end{equation}

Note that    a minimizer $v^n_p$   of the energy  ${\mathcal{E}}^{(n)}_p $ defined in \cite{hps}  such that  $v^n_p=g$ on $V^0$  could be different from a minimizer $u^n_p$ of the $p$-energy in \eqref{penergy} with the same boundary datum.

In \cite[Cor.2.4]{hps},  it is proved that, for $p$ fixed, any sequence  $\{v^n_p\}$   of the $p$-harmonic-extensions with respect to the energy   ${\mathcal{E}}^{(n)}_p$ on $V^n$ (such that $v^n_p=g$ on $V^0$) converges uniformly  as $n\rightarrow \infty$ to  a  function $v_p$  that is a $p$-harmonic extension of $g$ on $\cS$  with respect the  limit $p$-energy  defined in  \eqref{E}. \\
Passing to the limit for $p\to \infty$ we can prove that  any sequence  ${v_p}$ of  $p$-harmonic extensions with respect the $p$-energy ${\mathcal{E}}_p$ on $\cS $ such that $v_p=g$ on $V^0$ converges (up to a subsequence)   to a function $v:\cS\to\R$ uniformly in $\cS$
but up to now we are not able to prove that $v$ is a AMLE of $g$ to $\cS$.

\end{document}